\documentclass[brazil, 12pt]{amsart}
\ProvidesLanguage{portuges}
\usepackage[latin1]{inputenc}
\usepackage{indentfirst}
\usepackage{amsfonts}
\usepackage{graphicx}
\usepackage{amsmath}
\usepackage{amssymb}
\usepackage{latexsym}
\usepackage{amscd}
\usepackage[all]{xy}
\usepackage{multirow}

\newtheorem{theorem}{Theorem}[section]
\newtheorem{definition}[theorem]{Definition}
\newtheorem{lemma}[theorem]{Lemma}
\newtheorem{corollary}[theorem]{Corollary}
\newtheorem{proposition}[theorem]{Proposition}
\newtheorem{remark}[theorem]{Remark}
\newtheorem{example}[theorem]{Example}
\newenvironment{Proof}{\noindent{\bf Proof} }{\hfill $\square$}

\begin{document}

\title[Weierstrass Semigroup and Pure Gaps at several points on the $GK$ curve]{Weierstrass Semigroup and Pure Gaps at several points on the $GK$ curve}

\author{A. S. Castellanos and G. Tizziotti}

\maketitle

\begin{abstract}
We determine the Weierstrass semigroup $H(P_{\infty}, P_{1}, \ldots , P_{m})$ at se\-veral points on the $GK$ curve. In addition, we present conditions to find pure gaps on the set of gaps $G(P_{\infty}, P_{1}, \ldots , P_{m})$. Finally, we apply the results to obtain AG codes with good relative parameters.
\end{abstract}

\section{Introduction}

Curves with many rational points and a large automorphism group have been investigated for their applications in coding theory. In \cite{GK}, Giulietti and Korchmáros introduced a maximal curve, so called $GK$ curve, which is not a subcover of the corresponding Hermitian curve. The $GK$ curve is one of the rare examples of curves over a finite field where the automorphism group $Aut(GK)$ is rather large with respect to the genus. Another interesting fact about this curve is that the set of rational points splits into two non-trivial orbits, $\mathcal{O}_1$ and $\mathcal{O}_2$, and $Aut(GK)$ acts on $\mathcal{O}_1$ as $PGU(3,n)$ in its doubly transitive permutation representation, see \cite[Theorem 3.4]{GKcodes}. More recently, we can find applications of the $GK$ curve in coding theory, see \cite{zini}, \cite{CT_GK} and \cite{GKcodes}.

 As is known, Weierstrass semigroup is also an important tool in coding theory, see e.g. \cite{carvalho2}, \cite{GarciaKimLax} and \cite{gretchen3}. In this work, we determine the Weierstrass semigroup $H(P_{\infty}, P_1,\ldots, P_m)$ at $m+1$ points on $\mathcal{O}_1$, with $1 \leq m \leq |\mathcal{O}_1|$, where $P_{\infty}$ is the single point at infinity on $GK$. Our results were obtained using the concept of discrepancy, for given rational points $P$ and $Q$ on a curve $\mathcal{X}$, see Definition \ref{defi discrepancy}. This concept was introduced by Duursma and Park in \cite{duursma}, and it was our main tool for obtain the set $\Gamma(P_{\infty}, P_1,\ldots,P_m)$, called \textit{minimal generating set} of $H(P_{\infty}, P_{1}, \ldots , P_{m})$, see Theorem \ref{maintheorem}. In addition, we present conditions to find pure gaps on the set of gaps $G(P_{\infty},P_{1}, \ldots , P_{m})$.

This paper is organized as follows. Section 2 contains general results about Weierstrass semigroup and discrepancy, in addition to basic facts about AG codes and the $GK$ curve. In Section 3, we determine the minimal generating set for the Weierstrass semigroup $H(P_{\infty},P_1,\ldots, P_m)$ at points on the orbit $\mathcal{O}_1$ cited above. Finally, in Section 4 we present some results about pure gaps and AG codes over the $GK$ curve.

\section{Preliminaries}

We begin this section by introducing some notations that will be used in this work. Let $\mathcal{X}$ be a nonsingular, projective, geometrically irreducible curve of genus $g \geq 1$ defined over a finite field $\mathbb{F}_q$, let $\mathbb{F}_q(\mathcal{X})$ be the field of rational functions and $Div(\mathcal{X})$ be the set of divisors on $\mathcal{X}$. For $f \in \mathbb{F}_q(\mathcal{X})$, the divisor of $f$ will be denoted by $(f)$ and the divisor of poles of $f$ by $(f)_{\infty}$. For a divisor $G$ on $\mathcal{X}$, let $\mathcal{L}(G):= \{ f \in \mathbb{F}_{q}(\mathcal{X}) \mbox{ ; } (f) + G \geq 0 \} \cup \{ 0 \}$ be the Riemann-Roch space of $G$ and let $\dim(\mathcal{L}(G))$ be the dimension of $\mathcal{L}(G)$ as an $\mathbb{F}_{q}$-vector space. Let $\Omega(G)$ be the space of differentials $\eta$ on $\mathcal{X}$ such that $\eta=0$ or $div(\eta)\geq G$, where $div(\eta)= \sum_{P\in \mathcal{X}}ord_{P}(\eta)P$ and $ord_P(\eta)$ is the order of $\eta$ at $P$. As follows, we denote $\mathbb{N}_{0} = \mathbb{N} \cup \{0\}$, where $\mathbb{N}$ is the set of positive integers.

\subsection{Weierstrass semigroup and Discrepancy}
Let $P_{1}, \ldots , P_{m}$ be distinct rational points on $\mathcal{X}$. The set
$$
H(P_{1}, \ldots , P_{m}) = \{(a_{1}, \ldots, a_{m}) \in \mathbb{N}_{0}^ {m} \mbox{ ; } \exists f \in \mathbb{F}_q(\mathcal{X}) \mbox{ with } (f)_{\infty} = \sum_{i=1}^ {m} a_{i}P_{i} \}
$$
is called the \textit{Weierstrass semigroup} at the points $P_{1}, \ldots , P_{m}$. It is not difficult to see that the set $H(P_{1}, \ldots , P_{m})$ is a semigroup. An element in $\mathbb{N}_{0}^{m} \setminus H(P_{1}, \ldots , P_{m})$ is called \textit{gap} and the set $G(P_{1}, \ldots , P_{m}) = \mathbb{N}_{0}^{m} \setminus H(P_{1}, \ldots , P_{m})$ is called \textit{gap set} of $P_{1}, \ldots , P_{m}$.

Define a partial order $\preceq$ on $\mathbb{N}_0^m$ by $(n_1,\ldots,n_m)\preceq (p_1,\ldots,p_m)$ if and only if $n_i\leq p_i$ for all $i$, $1\leq i\leq m$.

For ${\bf u}_1,\ldots,{\bf u}_t\in \mathbb{N}_0^{m}$, where, for all $k$, ${\bf u}_k = (u_{k_{1}}, \ldots , u_{k_{m}})$, we define the \emph{least upper bound} ($lub$) of the vectors ${\bf u}_1,\ldots,{\bf u}_t$ in the following way: \[lub\{{\bf u}_1,\ldots,{\bf u}_t\}=(\max\{{ u_{1_1}},\ldots,{ u_{t_1}}\},\ldots, \max\{{ u_{1_m}},\ldots,{ u_{t_m}}\} )\in \mathbb{N}_0^{m}.\] For $\mathbf{n}=(n_1,\ldots,n_m)\in \mathbb{N}_0^{m}$ and $i \in \{ 1,\ldots , m\}$, we set
$$
\nabla_i (\mathbf{n}):=\{ (p_1, \ldots , p_m) \in H(P_{1}, \ldots, P_m) \mbox{ ; } p_i=n_i \}.
$$

\begin{proposition}\label{minimal}\cite[Proposition 3]{gretchen1}
Let $\mathbf{n}=(n_1,\ldots,n_m)\in \mathbb{N}_0^{m}$. Then $\mathbf{n}$ is minimal, with respect to $\preceq$, in $\nabla_i(\mathbf{n})$ for some $i$, $1 \leq i \leq m$, if and only if $\mathbf{n}$ is minimal in $\nabla_i(\mathbf{n})$ for all $i$, $1 \leq i \leq m$.
\end{proposition}

\medskip

\begin{proposition}\label{lubH}\cite[Proposition 6]{gretchen1}
Suppose that $1 \leq t \leq m \leq q$ and ${\bf u}_1,\ldots,{\bf u}_t\in H(P_{1}, \ldots , P_{m})$. Then $lub\{{\bf u}_1,\ldots,{\bf u}_t\} \in H(P_{1}, \ldots , P_{m})$.
\end{proposition}

\medskip

\begin{definition} \label{defi Gamma}
Let $\Gamma(P_{1})=H(P_1)$ and, for $m\geq 2$, define
$$
\Gamma(P_{1}, \ldots, P_{m}):=\{{\bf n}\in \mathbb{N}^m: \mbox{ for some } i, 1\leq i\leq m, {\bf n} \mbox{ is minimal in } \nabla_i ({\bf n})\}.
$$
\end{definition}

\medskip

\begin{lemma}\label{mPoints} \cite[Lemma 4]{gretchen1}
For $m \geq 2$, $\Gamma(P_{1}, \ldots, P_{m}) \subseteq G(P_1)\times\cdots\times G(P_m)\;.$
\end{lemma}

\medskip

In \cite{gretchen1}, Theorem 7, it is shown that, if $2\leq m \leq q$, then $H(P_1,\ldots,P_m)= $

\[ \left\{\begin{array}{cl} lub\{{\bf u}_1,\ldots,{\bf u}_m\}\in \mathbb{N}_0^m: & {\bf u}_i \in \Gamma(P_{1}, \ldots, P_{m}) \\& \mbox{ or } ( u_{i_1}, \ldots, u_{i_k}) \in \Gamma(P_{i_1}, \ldots , P_{i_k}) \\ & \mbox{ for some } \{i_1,\ldots,i_k\}\subset\{1,\ldots,m\}  \mbox{ such that } \\ & i_1<\cdots<i_k \mbox{ and } u_{i_{k+1}} = \cdots = u_{i_\ell}=0, \\ & \mbox{ where }  \{i_{k+1},\ldots,i_m\}\subset\{1,\ldots,\ell\} \setminus \{i_1,\ldots,i_k\} \end{array}\right\}\;.\]

Therefore, the Weierstrass semigroup $H(P_1,\ldots,P_m)$ is completely determined by $\Gamma(P_1,\ldots,P_m)$. In \cite{gretchen1}, Matthews called the set $\Gamma(P_1,\ldots,P_m)$ of \emph{minimal generating set} of $H(P_1,\ldots,P_m)$.

In \cite[Section 5]{duursma}, Duursma and Park introduced the concept of discrepancy as follows.

\begin{definition}\label{defi discrepancy}
A divisor $A \in Div(\mathcal{X})$ is called a \textit{discrepancy} for two rational points $P$ and $Q$ on $\mathcal{X}$ if $\mathcal{L}(A)\neq \mathcal{L}(A-P)=\mathcal{L}(A-P-Q)$ and $\mathcal{L}(A)\neq \mathcal{L}(A-Q)=\mathcal{L}(A-P-Q)$.
\end{definition}

%As a consequence of the previous lemma we have the following result.
%
%\begin{lemma}\cite[Lemma 5.1]{duursma} \label{lemma equiv discrepancy}
%A divisor $A$ is a discrepancy for the rational points $P$ and $Q$ if, and only if, the divisor $K+P+Q-A$ is a discrepancy for the rational points $P$ and $Q$, where $K$ is a canonical divisor.
%\end{lemma}

The next result relates the concept of discrepancy with the set $\Gamma(P_1,\ldots,P_m)$.

\begin{lemma} \label{lemma discrepancy} \cite[Lemma 2.6]{CT}
Let ${\bf n}=(n_1,\ldots,n_{m})\in H(P_1,\ldots,P_m)$. Then ${\bf n}\in \Gamma(P_1,\ldots,P_m)$ if and only if the divisor $A=n_1P_1+\cdots + n_m P_m$ is a discrepancy with respect to $P$ and $Q$ for any two rational points $P,Q\in \{P_1,\ldots,P_m\}$.
\end{lemma}

\subsection{AG codes} Let $D=P_1 + \ldots + P_n$ be a divisor on $\mathcal{X}$ such that $P_i \neq P_j$ for $i \neq j$. Let $G$ be another divisor on $\mathcal{X}$ such that $supp(D) \cap supp(G) = \emptyset$. Consider the maps $e_{v}: \mathcal{L}(G)  \rightarrow \mathbb{F}_{q}^n$ and $\varphi: \Omega(G-D) \rightarrow \mathbb{F}_{q}^n$ defined, respectively, by $e_v(f):=(f(P_1), \ldots , f(P_n))$ and $\varphi(\eta):=(res_{P_1}(\eta), \ldots , res_{P_n}(\eta))$, where $res_{P_i}(\eta)$ is the residue of $\eta$ at $P_i$, $i=1,\ldots,n$. We define the AG codes $C_{\mathcal{L}}(D,G)$ and $C_{\Omega}(D,G)$ as the images of the maps $e_v$ and $\varphi$, respectively. That is,
$$
\begin{array}{ccl}
C_{\mathcal{L}}(D,G) & := & \{ (f(P_{1}), \ldots , f(P_{n})) \mbox{ ; } f \in L(G)\};\\
C_{\Omega}(D,G) & := & \{ (res_{P_{1}}(\eta), \ldots , res_{P_{n}}(\eta)) \mbox{ ; } \eta \in \Omega (G-D) \}.
\end{array}
$$

The AG codes $C_{\mathcal{L}}(D,G)$ and $C_{\Omega}(D,G)$ are dual to each other. Let $[n,k,d]$ and $[n,k_{\Omega}, d_{\Omega}]$ be the length, dimension and minimum distance of $C_{\mathcal{L}}(D,G)$ and $C_{\Omega}(D,G)$, respectively. By Riemann-Roch Theorem we can estimate the parameters $[n,k,d]$ and $[n,k_{\Omega}, d_{\Omega}]$. In particular, if $2g-2 < deg(G) < n$, we have that $k=deg(G) - g + 1$, $d \geq n - deg(G)$,  $k_{\Omega}= n - deg(G) + g - 1$ and $d_{\Omega} \geq deg(G) - 2g + 2$, see e.g. \cite{vanlint2}. The right-hand side of the inequalities involving the minimum distance is known as the Goppa bound. One of the ways to obtain codes with good parameters is to find codes whose minimum distance have bounds better than the Goppa bound. In addition, another way is study codes over curves with many rational points, more specifically, codes arising from maximal curves. We remember that a curve $\mathcal{X}$ of genus $g$ over $\mathbb{F}_{q}$ is a \textit{maximal curve} if its number of rational points attains the Hasse-Weil upper bound, namely equals $2g\sqrt{q}+q+1$.

If $G=aQ$ for some rational point $Q$ on $\mathcal{X}$ and $D$ is the sum of all the other rational points on $\mathcal{X}$, then the AG codes $C_{\mathcal{L}}(D,G)$ and $C_{\Omega}(D,G)$ are called \textit{one-point AG codes}. Analogously, if $G= a_{1}Q_{1} + \cdots + a_{m}Q_{m}$, for $m$ distinct rational points on $\mathcal{X}$ and $D$ is the sum of all the other rational points on $\mathcal{X}$, then $C_{\mathcal{L}}(D,G)$ and $C_{\Omega}(D,G)$ are called $m$-\textit{point AG codes}. For more details about coding theory, see \cite{vanlint}, \cite{stichtenoth2} and \cite{vanlint2}.

\subsection{The $GK$ curve}\label{subsection GK} Let $q=n^3$, where $n \geq 2$ is a prime power. The $GK$ curve over $\mathbb{F}_{q^2}$ is the curve of $\mathbb{P}^{3}(\overline{\mathbb{F}}_{q^2})$ with affine equations

\begin{equation} \label{equation GK}
\left\{ \begin{array}{c}
  Z^{n^2-n+1} = Y h(X) \\
  X^{n} + X = Y^{n+1}\;\;\;,
\end{array} \right.
\end{equation}

where $\displaystyle h(X) = \sum_{i=0}^{n} (-1)^{i+1} X^{i(n-1)}$. We will denote this curve simply by $GK$. The curve $GK$ is absolutely irreducible, nonsingular, has $n^8 - n^6 + n^5 + 1$ $\mathbb{F}_{q^2}$-rational points, a single point at infinity $P_{\infty}=(1:0:0:0)$ and its genus is $g = \dfrac{1}{2} (n^3 + 1)(n^2 - 2) + 1$. The $GK$ curve has an important properties as it lies on the Hermitian surface $\mathcal{H}_3$ with affine equation $X^{n^3} + X = Y^{n^3 + 1} + Z^{n^3+1}$; it is a maximal curve and, for $q>8$, $GK$ is the only know curve that is maximal but not $\mathbb{F}_{q^2}$-covered by the Hermitian curve $\mathcal{H}_2$ defined over $\mathbb{F}_{q^2}$ and its automorphism group $Aut(GK)$ has size $n^3(n^3+1)(n^2-1)(n^2-n+1)$ which turns out to be very large compared to the genus $g$.

Let $GK(\mathbb{F}_{q^2})$ be the set of $\mathbb{F}_{q^2}$-rational points of $GK$. For $j=1,\ldots,n$, let $P_j = (a_j,0,0)\in GK(\mathbb{F}_{q^2})$ such that $a_{j}^{n}+a_j=0$, and, for $\ell=1,\ldots,n^3 - n$, let $Q_{\ell}=(a_{\ell},b_{\ell},0)\in GK(\mathbb{F}_{q^2})$ such that $b_{\ell}\neq 0$ and $a_{\ell}^{n}+a_{\ell}=b_{\ell}^{n+1}$. In the following, $P_{\infty}$, $P_j$, for $j=1,\ldots,n$, and $Q_{\ell}$, for $\ell=1,\ldots,n^3 - n$, will be the points given above.

Since $P_{\infty}$ is a single point at infinity of $GK$ and the function field $\mathbb{F}_{q^2}(GK)$ is $\mathbb{F}_{q^2}(x,y,z)$ with $z^{n^2-n+1} = y h(x)$ and $x^n + x = y^{n+1}$ we have that
\begin{equation} \label{divisors}
\begin{array}{rcl}
 (x-a_j) & = & (n^3 + 1)P_j - (n^3+1)P_{\infty}, \mbox{ for } j=1,\ldots,n ; \\
 (y) & = & \sum_{j=1}^{n}(n^2 - n + 1)P_j - n(n^2-n+1)P_{\infty}; \\
 (z) & = & \sum_{j=1}^{n}P_j + \sum_{\ell = 1}^{n^3-n} Q_{\ell} - n^3 P_{\infty} .
\end{array}
\end{equation}

\begin{theorem} \label{theorem bitransitive} \cite[Theorem 3.4]{GKcodes}
The set of $\mathbb{F}_{q^2}$-rational points of $GK$ splits into two orbits under the action of $Aut(GK)$. One orbit, say $\mathcal{O}_1$, has size $n^3+1$ and consists of the points $P_j$ and $Q_{\ell}$ as above together with the infinite point $P_{\infty}$. The other orbit has size $n^3(n^3+1)(n^2-1)$ and consists of the points $P=(a,b,c) \in GK(\mathbb{F}_{q^2})$ with $c \neq 0$. Furthermore, $Aut(GK)$ acts on $\mathcal{O}_1$ as $PGU(3,n)$ in its doubly transitive permutation representation.
\end{theorem}

\begin{proposition} \label{proposition H(P)} \cite[Proposition 3.1]{GKcodes}
Let $P_j$ and $Q_{\ell}$ be as above. Then, $H(P_{\infty}) = H(P_j) = H(Q_{\ell}) = \langle n^3 - n^2 + n , n^3 , n^3 + 1 \rangle$, for each $j=1,\ldots,n$, and $\ell=1,\ldots,n^3-n$.
\end{proposition}

For more details about the $GK$ curve, see \cite{GK}.

\section{The Weierstrass semigroup at certain $m+1$ points on $GK$ curve}

In this section we will determine the Weierstrass semigroup $H(P_{\infty},P_1,\ldots, P_m)$, for $1\leq m\leq n$. To simplify the notation, we will denote $a=n^2-n+1$, $b=n^3$ and $c=n^3+1$. By the divisors of the rational functions $(x-a_j)$, $y$ and $z$ given in (\ref{divisors}), we have the following equivalences

\begin{equation}\label{eqx}
cP_j \sim cP_\infty\;,
\end{equation}
\begin{equation}\label{eqy}
aP_1+\cdots+aP_{n}\sim naP_\infty \;,
\end{equation}
\begin{equation}\label{eqz}
P_1+\cdots+P_{n}+Q_1+\cdots +Q_{n^3-n}\sim bP_\infty\;.
\end{equation}

Let $1\leq m\leq n$ and let $1\leq k\leq a$, $0\leq i\leq n$ and $j_s\geq 0$ be integers such that $$\left(n^2-m-\sum_{s=1}^m j_s\right)c-ina-kb>0\;.$$

So, the divisor
\begin{equation} \label{Alinha}
A'=((n^2-m)c-ina-kb)P_\infty+\sum_{s=1}^m(ia+k)P_s
\end{equation}

is effective and using (\ref{eqx}), (\ref{eqy}) and (\ref{eqz}) we have that

\begin{equation}\label{eq divisors equiv}
((n^2-m-\sum_{s=1}^m j_s)c-ina-kb)P_\infty +\sum_{s=1}^m(j_s c+ia+k)P_s\sim A'.
\end{equation}

The following lemma is important to show that a divisor is a discrepancy.

\begin{lemma}\cite[Noether's Reduction Lemma]{fulton} \label{lemma noether}
Let $D$ be a divisor, $P \in \mathcal{X}$ and let $K$ be a canonical divisor. If $\dim(\mathcal{L}(D))>0$ and $\dim(\mathcal{L}(K-D-P)) \neq \dim(\mathcal{L}(K-D))$, then $\dim(\mathcal{L}(D+P))=\dim(\mathcal{L}(D))$.
\end{lemma}

\begin{proposition} \label{prop discrepancy}
The divisor $A'$ is a discrepancy with respect to $P$ and $Q$ for any two distinct points $P,Q\in\{P_{\infty},P_1,\ldots,P_m\}$.
\end{proposition}
\begin{Proof}
First, note that the equivalence of effective divisors in (\ref{eq divisors equiv}) gives a rational
function $f \in \mathcal{L}(A')$ with pole divisor equal to $A'$. Thus, $\mathcal{L}(A') \neq \mathcal{L}(A'-P)$ for all $P \in \{ P_{\infty},P_1,\ldots,P_m \}$.

Now, we must prove that $\mathcal{L}(A'-P) = \mathcal{L}(A'-P-Q)$ and $\mathcal{L}(A'-Q) = \mathcal{L}(A'-Q-P)$, for all $P,Q\in\{P_{\infty},P_1,\ldots,P_m\}$. By Lemma \ref{lemma noether}, it suffices to prove that $\mathcal{L}(K-A'+P)\neq\mathcal{L}(K-A'+P+Q)$, where $K$ is a canonical divisor. Taking $K=(n^2-2)c P_{\infty}$, we have that
$$
\begin{array}{ll}
K+P+Q-A'& \displaystyle = (n^2-2)cP_\infty+P+Q-((n^2-m)c-ina-kb)P_\infty - \sum_{s=1}^m(ia+k)P_s \\
         & \displaystyle = (c(m-2)+ina+kb)P_\infty + P + Q- \sum_{s=1}^m(ia+k)P_s
\end{array}
$$
If $P_{\infty} \in \{ P, Q \}$, without loss of generality, assume that $P=P_{\infty}$ and $Q=P_1$. Thus,
$$
K+P+Q-A'=(c(m-1)+ina+(k-1)b)P_\infty - (ia+k-1)P_1 - \sum_{s=2}^m (ia+k)P_s
$$

and we have that \[z^{k-1}y^i(x-a_2)\cdots (x-a_m)\in \mathcal{L}(K+P+Q-A')\setminus \mathcal{L}(K+Q-A')\;.\] So, $\mathcal{L}(A'-Q) = \mathcal{L}(A'-P-Q)$.
Since $\mathcal{L}(A') \neq \mathcal{L}(A'-Q) = \mathcal{L}(A'-P-Q)$ and $\mathcal{L}(A') \neq \mathcal{L}(A'-P)$, it follows that $\mathcal{L}(A'-P) = \mathcal{L}(A'-P-Q).$

If $P_{\infty} \not \in \{ P, Q \}$, we can suppose that $P=P_1$ and $Q=P_2.$ In this case, we have that
$$
z^{k-1}y^i(x-a_3)\cdots (x-a_m)\in \mathcal{L}(K+P+Q-A')\setminus \mathcal{L}(K+Q-A')\;.
$$ As above, it follows that $\mathcal{L}(A'-Q) = \mathcal{L}(A'-P-Q)$ and that $\mathcal{L}(A'-P) = \mathcal{L}(A'-P-Q).$

Therefore, $A'$ is a discrepancy with respect to $P$ and $Q$ for any two distinct points $P,Q\in\{P_{\infty},P_1,\ldots,P_m\}$.
\end{Proof}

\begin{remark} \label{remark 1}
From (\ref{eq divisors equiv}) and Definition \ref{defi discrepancy} follows that the divisor $((n^2-m-\sum_{s=1}^m j_s)c-ina-kb)P_\infty +\sum_{s=1}^m(j_s c+ia+k)P_s$ is also a discrepancy with respect to $P$ and $Q$ for any two distinct points $P,Q\in\{P_{\infty},P_1,\ldots,P_m\}$.
\end{remark}

\begin{theorem}\label{maintheorem}
Let $a$, $b$, $c$, $P_{\infty}, P_1, \ldots, P_{m}$ be as above. For $1 \leq m \leq n$, let
$$
\begin{array}{ll}
\Gamma_{m+1} & =\left\{  \left( (n^2-m-\sum_{s=1}^m j_s)c-ina-kb, j_1 c+ia+k, \ldots , j_m c+ia+k \right) \right.; \\
                             & \left. 1\leq k\leq a \mbox{, } 0\leq i\leq n \mbox{, } j_s\geq 0 \mbox{ and } \left(n^2-m-\sum_{s=1}^m j_s\right)c-ina-kb>0 \right\}.
\end{array}
$$

Then, $\Gamma(P_{\infty},P_1,\ldots,P_{m}) = \Gamma_{m+1}$.
\end{theorem}
\begin{Proof}
By Proposition \ref{prop discrepancy}, $A'$ is a discrepancy with respect to $P$ and $Q$ for any two distinct points $P,Q\in\{P_{\infty}, P_1,\ldots,P_m\}$. By Remark \ref{remark 1}, the divisor
%\begin{equation}\label{divisor A}
$A=((n^2-m-\sum_{s=1}^m j_s)c-ina-kb)P_\infty +\sum_{s=1}^m(j_s c+ia+k)P_s$
%\end{equation}
is a discrepancy with respect to $P$ and $Q$ for any two distinct points $P,Q\in\{P_{\infty}, P_1,\ldots,P_m\}$. Therefore, by Lemma \ref{lemma discrepancy}, we have that $\Gamma_{m+1} \subseteq \Gamma(P_{\infty}, P_1,\ldots,P_{m})$.

Next, we show that $\Gamma(P_{\infty},P_1,\ldots,P_{m})\subseteq \Gamma_{m+1}$. Let $\mathbf{n} = (n_0, n_1 , \ldots , n_m)\in \Gamma(P_{\infty}, P_1,\ldots,P_{m})$. By Definition \ref{defi Gamma} and Proposition \ref{minimal}, follows that $\mathbf{n}$ is minimal in $\nabla_r(\mathbf{n})$ for all $r$, $1 \leq r \leq m+1$. By Lemma \ref{mPoints},
$\mathbf{n} = (n_0, n_1, \ldots, n_m)\in G(P_{\infty}) \times G(P_1)\times\cdots\times G(P_m)$.

Note that $H(P_s)= \langle an , b , c \rangle$, for all $1 \leq s \leq m$, and then $n_s=j_s c + i_s a + k_s$, for some $j_s \geq 0$, $0 \leq i_s \leq n$ and $1 \leq k_s \leq a$. Let
$$f=\dfrac{y^{n-i} z^{a-k}}{(x-a_1)^{j_1 + 1} \ldots (x-a_m)^{j_m + 1}}.$$
Then, $(f)_{\infty} = ((n^2-m-\sum_{s=1}^m j_s)c-ina-kb)P_{\infty} + (j_1 c+ia+k)P_1 + \cdots + (j_m c+ia+k)P_m$. Now, by Remark \ref{remark 1}, $(f)_{\infty}$ is a discrepancy with respect to $P$ and $Q$ for any two distinct points $P,Q\in\{P_{\infty},P_1,\ldots,P_m\}$ and so, by Lemma \ref{lemma discrepancy},

 $\mathbf{f} = ( (n^2-m-\sum_{s=1}^m j_s)c-ina-kb, j_1 c+ia+k, \ldots , j_m c+ia+k ) \in \Gamma(P_{\infty}, P_1, \ldots , P_{m})$.

 Thus, $\mathbf{f} \in \nabla_r ({\bf n})$, for some $1 \leq r \leq m+1$, and, by Proposition \ref{minimal}, it follows that $\mathbf{f}$ is minimal in $\nabla_r(\mathbf{n})$ for all $r$, $1 \leq r \leq m+1$. Therefore, by minimality of $\mathbf{f}$ and $\mathbf{n}$, we have that $\mathbf{f}=\mathbf{n}$ and $\Gamma(P_{\infty}, P_1,\ldots,P_{m}) \subseteq \Gamma_{m+1}$.

%$$
%\begin{array}{ll}
%n_0 & = j_0 c + i_0 a + k_0 \\
%       & = j_0 c + i_0 a + i_0 n a - i_0 n a + k_0 + k_0 b - k_0 b \\
%       & = j_0 c + i_0 a (n +1) - i_0 n a + k_0 c - k_0 b \\
%       & = (j_0 + k_0 + i_0) c - i_0 n a - k_0 b \\
%\end{array}
%$$
\end{Proof}

\begin{example}
For $n=2$, we have $a=3, b=8,c=9$, and the curve $GK$ with affine equations

\begin{equation} \label{equation GK n 2}
\left\{ \begin{array}{c}
  Z^{3} = Y (1+X+X^2) \\
  X^{2} + X = Y^{3}
\end{array} \right.
\end{equation}

In this case, the genus $g=10$ and, by Proposition \ref{proposition H(P)}, $H(P_1)=H(P_2)=H(P_{\infty}) = \langle 6,8,9 \rangle$ and then $G(P_{0}) = G(P_{\infty}) = \{ 1,2,3,4,5,7,10,11,13,19 \}$. We have the following divisors
\[
\begin{array}{ccl}
(x-a_\ell)& = & 9P_\ell-9P_\infty\,; \\
(y)& = & 3P_1+3P_2-6P_\infty\,; \\
(z)& = & P_1+P_2+Q_1+\cdots + Q_6-8P_\infty\,,
\end{array}
\]

where $\ell=1,2$. For this curve, taking $m=1$, by Theorem \ref{maintheorem}, we have that $\Gamma(P_\infty,P_1)=\{(9(3-j_1)-6i-8k,9j_1+3i+k):0\leq i \leq 2, 1\leq k \leq 3, j_1\geq 0 \mbox{ and } 9(3-j_1)-6i-8k >0\}$,
therefore $$\Gamma(P_\infty,P_1)=\{(1,19),(2,11),(3,3),(4,13),(5,5),(7,7),(10,10),(11,2),(13,4),(19,1)\}\;.$$

Taking $m=2$, we have that
$$\Gamma(P_\infty,P_1,P_2)=\{(9(2-j_1-j_2)-6i-8k, 9j_1+3i+k,9j_2+3i+k):$$ $$ 0\leq i \leq 2, 1\leq k \leq 3, j_1,j_2\geq 0 \mbox{ and } 9(2-j_1-j_2)-6i-8k >0\}$$ then
\[\Gamma(P_\infty,P_1,P_2)=\{(10,1,1),(1,1,10),(1,10,1),(2,2,2),(4,4,4)\}\;.\]

\end{example}

Other different examples can be found explicitly in www.alonso.prof.ufu.br/Example.pdf.

\section{Pure gaps and codes over the $GK$ curve}

An element $(\alpha_1, \ldots,\alpha_m) \in G(P_{1}, \ldots , P_{m})$ is called \textit{pure gap} if $\dim (\mathcal{L}(\sum_{i=1}^m \alpha_i P_i)) = \dim (\mathcal{L}(\sum_{i=1}^m \alpha_i P_i - P_j))$, for all $j=1,\ldots,m$. This concept was introduced by M. Homma and S.J. Kim in \cite{homma}. In \cite{carvalho2}, C. Carvalho and F. Torres used the concept of pure gaps to obtain codes whose minimum distance have bounds better than the Goppa bound.

\medskip

\begin{theorem} \label{teorema Cicero} \cite[Theorem 3.3]{carvalho2}
Let $Q_1, \ldots, Q_n, P_1, \ldots , P_m$ be distinct $\mathbb{F}_{q}$-rational points of $\mathcal{X}$ and assume that $m \leq q$. Let $(\alpha_1, \ldots, \alpha_m), (\beta_1, \ldots, \beta_m)\in \mathbb{N}_0^{m}$ and set $D=Q_1 + \cdots + Q_n$ and $G=\sum_{i=1}^m (\alpha_i + \beta_i - 1)P_i$. Let $d_{\Omega}$ be the minimum distance of the code $C_{\Omega}(D,G)$. If  $(\alpha_1, \ldots, \alpha_m)$ and $(\beta_1, \ldots, \beta_m)$ are pure gaps at $P_1, \ldots , P_m$, then $d_{\Omega} \geq deg(G) - (2g - 2) + m$, where $g$ is the genus of $\mathcal{X}$.
\end{theorem}

Using the concept of discrepancy we have the following result to obtain pure gaps.

\begin{proposition}\label{lemma pure gaps}
Let $A=\sum_{\ell=0}^m a_{\ell}P_{\ell}$, where $(a_0, a_1,\ldots,a_m) \in \Gamma(P_0, P_1, \ldots, P_m)$. Let $\ell \in \{0,1,\ldots,m\}$, if $\mathcal{L}(A - P_{\ell}) = \mathcal{L}(A - 2P_{\ell})$, then

$(a_0,a_1,\ldots,a_{{\ell}-1},a_{\ell} - 1, a_{\ell+1},\ldots,a_m)$ is a pure gap of $H(P_0, P_1, \ldots , P_m)$.
\end{proposition}

\begin{proof}
In fact, by Lemma \ref{lemma discrepancy}, the divisor $A$ is a discrepancy with respect to $P$ and $Q$ for any two distinct points $P,Q\in\{P_0,P_1,\ldots,P_m\}$. So, $\mathcal{L}(A-P_{\ell})=\mathcal{L}(A-P_{\ell}-Q)$ for any $Q\in\{P_0,P_1,\ldots,P_{m}\} \setminus \{P_{\ell}\}$. Thus, if $\mathcal{L}(A - P_{\ell}) = \mathcal{L}(A - 2P_{\ell})$, by definition of pure gap, follows that $(a_0,a_1,\ldots,a_{\ell-1},a_{\ell} - 1, a_{\ell+1},\ldots,a_m)$ is a pure gap of $H(P_0, P_1, \ldots , P_m)$.
\end{proof}

For Corollary \ref{exemplo pure gaps} and Lemma \ref{lemma pure gaps 2} in the following, consider the $GK$ curve over $\mathbb{F}_{n^6}$ with genus $g$. We remember that $a=n^2-n+1$, $b=n^3$, $c=n^3+1$ and $1\leq m \leq n$, and that $P_{\infty},P_1,\ldots,P_m$ are the rational points given in Section 3.

\begin{corollary} \label{exemplo pure gaps}
If $2 \leq k \leq a$, then $((n^2 - m)c - kb , k, \ldots , k, k-1)$ is a pure gap of Weierstrass semigroup $H(P_{\infty},P_1,\ldots,P_m)$ on $GK$.
\end{corollary}
\begin{proof}
In fact, first note that $((n^2 - m)c - kb , k, \ldots , k, k) \in \Gamma(P_{\infty}, P_1, \ldots, P_m)$ by taking $i=0$ and $j_s=0$, for all $s=1,\ldots,m$, in the Theorem \ref{maintheorem}. Let $A=((n^2-m)c - kb)P_{\infty} + \sum_{s=1}^m k P_s$. By the previous Lemma, we must prove that $\mathcal{L}(A - P_m) = \mathcal{L}(A - 2P_m)$. Let $K=(2g-2)P_{\infty}=(n^2-2)cP_{\infty}$ be a canonical divisor. Note that $z^{k-2}(x-a_1)\ldots(x-a_{m-1}) \in \mathcal{L}(K-A + 2P_m) \setminus \mathcal{L}(K-A + P_m)$ and so $\mathcal{L}(K-A + P_m) \neq \mathcal{L}(K-A + 2P_m)$. Thus, by Lemma \ref{lemma noether}, it follows that $\mathcal{L}(A - P_m) = \mathcal{L}(A - 2P_m)$.
\end{proof}

\begin{proposition}\label{lemma pure gaps 2}
Let $\alpha < 2g-1$ and  $(\alpha,1,1,\ldots,1) \in G(P_{\infty}, P_1, \ldots , P_m)$. If

i) $\exists \lambda , \beta , \gamma \in \mathbb{N}_0$, with $\lambda \geq m$, such that $\lambda c + \beta a n + \gamma b = 2g-1-\alpha$, or

ii) $2g-1-\alpha \geq (m-1)c$ and $\exists \beta, \gamma \in \mathbb{N}_0$ such that $ \beta a n + \gamma b = 2g-1-\alpha$,

then $(\alpha,1,1,\ldots,1)$ is a pure gap.
\end{proposition}

\begin{proof}
Let $\alpha < 2g-1$ and  $(\alpha,1,1,\ldots,1) \in G(P_{\infty}, P_1, \ldots , P_m)$. Consider the divisor $A=\alpha P_{\infty} + P_1 + \cdots + P_m$ and the canonical divisor $K = (2g-2)P_{\infty}$. We will see that, if the conditions i) or ii) as above are satisfied, then $\mathcal{L}(K-A) \neq \mathcal{L}(K-A+P_{\infty})$ and $\mathcal{L}(K-A) \neq \mathcal{L}(K-A+P_{i})$, for all $i=1,\ldots,m$. Thus, Lemma \ref{lemma noether}, we conclude that $(\alpha,1,1,\ldots,1)$ is a pure gap.

First, suppose that $\exists \lambda , \beta , \gamma \in \mathbb{N}_0$, with $\lambda \geq m$, such that $\lambda c + \beta a n + \gamma b = 2g-1-\alpha$. Since $\lambda \geq m$, we can write $\lambda = \lambda_1 + \cdots + \lambda_m$, with $\lambda_i \geq 1$ for all $i=1,\ldots,m$. So, $(x-a_1)^{\lambda_1}\ldots (x-a_m)^{\lambda_m}y^{\beta} z^{\gamma} \in \mathcal{L}(K-A+P_{\infty}) \setminus \mathcal{L}(K-A)$ and $\dfrac{(x-a_1)\ldots (x-a_m)}{x-a_i} \in \mathcal{L}(K-A+P_{i}) \setminus \mathcal{L}(K-A)$, for all $i=1,\ldots,m$. Therefore, we have that $(\alpha,1,1,\ldots,1)$ is a pure gap.

Now, suppose that $2g-1-\alpha \geq (m-1)c$ and $\exists \beta, \gamma \in \mathbb{N}_0$ such that $ \beta a n + \gamma b = 2g-1-\alpha$. So, $y^{\beta} z^{\gamma} \in \mathcal{L}(K-A+P_{\infty}) \setminus \mathcal{L}(K-A)$ and $\dfrac{(x-a_1)\ldots (x-a_m)}{x-a_i} \in \mathcal{L}(K-A+P_{i}) \setminus \mathcal{L}(K-A)$, for all $i=1,\ldots,m$. Therefore, we have that $(\alpha,1,1,\ldots,1)$ is a pure gap.
\end{proof}

Let us remember that given a code $C$ with parameters $[n,k,d]$, we define its \textit{information rate} by $R=k/n$ and its \textit{relative minimum distance} by $\delta=d/n$. These parameters allows us to compare codes with different length. In the following example we get a code that have better relative parameters than the corresponding one-point code given in \cite[Table IV]{GKcodes}.

\begin{example}
Consider the $GK$ curve over $\mathbb{F}_{3^6}$ with affine equations \[Z^7=Y(2+X^2+2X^4+X^6)\;,\qquad X^3+X=Y^4\;.\] This curve has $6076$ $\mathbb{F}_{3^6}$-rational points and genus $g=99$. By Corollary \ref{exemplo pure gaps} and Proposition \ref{lemma pure gaps 2}, respectively, it follows that $(142,2,2,1)$ and $(155,1,1,1)$ are pure gaps at $P_{\infty}, P_1, P_2, P_3$. By Theorem \ref{teorema Cicero}, the $4$-point code $C_{\Omega}(D,296P_{\infty}+2P_{1}+2P_2+P_3)$ of dimension $k_{\Omega}=5869$ has minimum distance $d_{\Omega} \geq 109$.
\end{example}


\begin{thebibliography}{00}

%% \bibitem{label}
%% Text of bibliographic item

%\bibitem{arbarello} E. Arbarello, M. Cornalba, P. Griffiths, J. Harris, {\em Geometry of Algebraic Curves}, Berlin, Germany: Springer-Verlag, 1985.

%\bibitem{kummer} M. Abd\'{o}n, H. Borges and L. Quoos, {\em Weierstrass points on Kummer extensions}, (Nov. 2015) [Online]. Available: http://arxiv.org/abs/1308.2203.

%\bibitem{ballico} E. Ballico, {\em Weierstrass points and Weierstrass pairs on algebraic curves}, Int. J. Pure Appl. Math., 2 (2002), 427-440.

\bibitem{zini} D. Bartoli, M. Montanucci and G. Zini, {\em Multi point AG codes on the GK maximal curve}, Designs, Codes and Cryptography, to appear.

%\bibitem{carvalho} C. Carvalho and T. Kato, {\em On Weierstrass semigroup and sets: a review with new results}, Geom. Dedicata 139 (2009), 139-195

\bibitem{carvalho2} C. Carvalho e F. Torres, {\em On Goppa codes and Weierstrass gaps at several points}, Des. Codes Cryptogr., 35(2) (2005), 211-225.

\bibitem{CT_GK}
A. S. Castellanos, G. Tizziotti, {\em Two-Point AG Codes on the GK Maximal Curves}. IEEE Transactions on Information Theory, v. 62, p. 681-686, 2016.

\bibitem{CT}
A. S. Castellanos, G. Tizziotti,
{\em On Weierstrass semigroup at $m$ points on curves of the form $f(y) = g(x)$}, to appear in Journal of Pure and Applied Algebra.


\bibitem{duursma} I. Duursma and S. Park, {\em Delta sets for divisors supported in two points}, Finite Fields and Their Applications, 18 (5), 2012, 865-885.

\bibitem{GKcodes} S. Fanali and M. Giulietti, {\em One-point AG Codes on the GK Maximal Curves}, IEEE Trans. on Information Theory, vol. 56, no. 1, pp. 202 - 210, 2010.

\bibitem{fulton} W. Fulton, {\em Algebraic Curves: an introduction to Algebraic Geometry}, Addison Wesley, 1969.

\bibitem{GarciaKimLax} A. Garcia, S. J. Kim, and R. F. Lax, {\em Consecutive Weierstrass gaps
and minimum distance of Goppa codes}, J. Pure Appl. Algebra, 84 (1993), 199-207.

\bibitem{GK} M. Giulietti and G. Korchmáros, {\em A new family of maximal curves over a finite field}, Mathematische Annalen, vol. 343, pp. 229 - 245, 2009.


%\bibitem{stichtenoth} J. P. Hansen e H. Stichtenoth, {\em Group
%codes on certain curves with many rational points}, AAECC 1, 67-77 (1990).

\bibitem{vanlint} T. H{\o}holdt, J. van Lint, R. Pellikaan, {\em Algebraic geometry
codes}, V.S. Pless, W.C. Huffman (Eds.), Handbook of Coding Theory, v. 1, Elsevier, Amsterdam, 1998.

\bibitem{homma} M. Homma e S. J. Kim, {\em Goppa codes with Weierstrass
pairs}, J. Pure Appl. Algebra, 162, 2001, 273-290.

%\bibitem{homma2} M. Homma, {\em The Weierstrass semigroup of a pair of points on a curve},
%Arch. Math., 67, 1996, 337-348.

%\bibitem{kim} S.J. Kim, {\em On index of the Weierstrass semigroup of a pair of points on a
%curve}, Arch. Math., 62, 1994, 73-82.

%\bibitem{kondo} S. Kondo, T. Katagiri and T. Ogihara {\em Automorphism groups of one-point codes from the curves $y^q + y = x^{q^r+1}$}, IEEE Trans Inform Theory, 47, 2001, 2573 - 2579.

%\bibitem{gretchen} G. L. Matthews, {\em Weierstrass pairs and minimum distance of Goppa
%codes}, Designs, Codes and Cryptography, 22, 2001, 107-121.

\bibitem{gretchen1} G. L. Matthews, {\em The Weierstrass semigroup of an $m$-tuple of collinear points on a Hermitian curve}, Lecture note in Comput. Sci., Spinger, Berlin, 2948, 2004, 12 - 24.

\bibitem{gretchen3} G. L. Matthews, {\em Weierstrass semigroups and codes from a quotient of the Hermitian curve}, Designs, Codes and Cryptography, 37, 2005, 473-492.

%\bibitem{gretchen2} G. L. Matthews and J. D. Peachey, {\em Minimal generating sets of Weierstrass semigroups of certain $m$-tuples on the norm-trace function field}, Contemporary Mathematics,1 518, 2010, 315-326.

%\bibitem{tizziotti}
%C. Munuera, G. Tizziotti and F. Torres, {\em Two-Points Codes on Norm-Trace Curves},
%Second International Castle Meeting, ICMCTA 2008 (A. Barbero Ed.), Lecture Notes in Comput. Sci., Springer-Verlag Berlin Heidelberg, 5228, 2008, 128-136.

%\bibitem{st}
%A. Sep\'ulveda, G. Tizziotti,
%{\em Weierstrass semigroup and codes over the curve $y^q + y = x^{q^r+1}$}, Advances in Mathematics of Communications, 8, 2014, 67-72.


%\bibitem{rosales}
%J. C. Rosales, {\em Fundamental gaps of numerical semigroups generated by two elements},
%Linear Algebra and its Applications, 405, 2005, 200-208

\bibitem{stichtenoth2} H. Stichtenoth, {\em Algebraic Function Fields and Codes}, Berlin, Germany: Springer, 1993.

\bibitem{vanlint2} J. H. van Lint, {\em Introduction to Coding Theory}, New York: Springer 1982.

%\bibitem{stichtenoth3} H. Stichtenoth, {\em Uber die Automorphismengruppe eines algebraischen
%Funktionenkorpers von Primzahlcharakteristik, Tail II }, Arch. Math., 24 (1973), 615-631.

\end{thebibliography}
\end{document}